\documentclass{ifacconf}

\usepackage{chris_style_sheet_ifac}

\begin{document}
\begin{frontmatter}

\title{Initial Condition Independent Stabilisability of Switched Affine Systems}

\author{Christopher Townsend$^*$} 
\author{Maria M. Seron$^{**}$} 

\address[First]{Unaffiliated (e-mail: chris.townsend@newcastle.edu.au).}
\address[Second]{School of Engineering, The University of Newcastle, Australia (e-mail: maria.seron@newcastle.edu.au)}

\begin{abstract}
We have previously demonstrated that a switched affine system is stabilisable independently of the initial condition, i.e. there exists an asymptotically stabilising switching function which is the same for all initial conditions, if and only if there exists a stable convex combination of the sub-system matrices. This result was proven by constructing a stabilising switching function of \emph{unbounded} switching frequency. The current paper proves that there exists a switching function with \emph{bounded} switching frequency which stabilises a switched affine system independent of its initial condition. 
\end{abstract}


\end{frontmatter}

\section{Introduction}

A switched system is a time-varying dynamic system consisting of a number of sub-systems and a switching function that specifies which sub-system controls the system dynamics at each time $t$. Switched systems are typically characterised by discontinuous dynamics at each switching instant. We consider linear and affine switched systems. In each case the systems are given by:
\begin{align}
	\label{eq:switchedaffine}
	\dot x = A_i x + b_i, \quad i=1,\dots, m,
\end{align}
for some natural number $m$, where $b_i \in \R ^{n \times 1}$ is non-zero for at least one $i$ if the system is affine and $b_i = 0$ in \Cref{eq:switchedaffine} for all $i$ if the system is linear i.e.
\begin{align}
    \label{eq:linsys}
    \dot x = A_i x
\end{align}

A switched system is \emph{stable} if it is Lyapunov stable \citep{khal92} for any switching function and is \emph{stabilisable} if there exists at least one switching function such that the system is Lyapunov stable. All stable systems are stabilisable. However, the converse is not true. Indeed, for a stabilisable system, the switching function $\sigma(x, t)$ may depend on the initial condition $x(0)$. Here we investigate \emph{when there is a stabilising switching function which is independent of the initial condition}. If such an \emph{Initial Condition Independent} (ICI) switching function exists we say the system is \emph{ICI stabilisable} -- \cite{town20}.

\begin{defn}[ICI Stabilisable]
A switched system is \emph{ICI stabilisable} if there exists a switching function $\sigma(t)$ such that 
\[
	\lim_{t \to \infty} x(t) = 0
\]%
for all initial conditions $x(0)$.
\end{defn}%

i.e. the switched system is stabilised by a time-varying state-independent switching function. A number of necessary conditions for the stability and sufficient conditions for the stabilisability of a switched system are known, see \cite{deae10}, \cite{libe03}, \cite{lin09} and \cite{shor07}. One such condition is the existence of a stable convex combination of the sub-systems i.e. the existence of a matrix $A \in co \l (\mathcal{A} \r)$ -- the convex hull of the set $\l \{ A_1, \cdots, A_m \r \}$ -- such that $Re (\lambda_j) < 0$ for all eigenvalues $\lambda_j$ of $A$. The existence of a stable convex combination is necessary for the stability of a switched linear system \citep{libe03} but is only a sufficient, and conservative, condition for the stabilisability of a switched linear system. Indeed, if there is a stable convex combination then the switched system satisfies the more general Lyapunov-Metzler inequalities \citep{gero06}. Using the $S$-Lemma, \cite{fero96} proved that the existence of a stable convex combination is only necessary and sufficient for the stabilisability of a switched linear system if the system has only two sub-systems. 

In \cite{town20}, it was proven that a switched affine system is ICI stabilisable if and only if there is a stable convex combination, $A \in co (\mathcal{A})$, of the system matrices. The proof relies on constructing a switching function -- the \emph{norm-minimising} switching function. 

\begin{defn}[Norm-minimising]
\label{defn:normmini}
The switching function $\sigma'(x,t)$ is norm-minimising if:
\[
	\sigma' \in \l \{ i : \frac{d}{dt} \l \| x_i \r \| \leq \frac{d}{dt} \l \| x_j \r \| \text{for all } j \neq i \r \}
	\]%
where $x_i$ is the solution to \eqref{eq:linsys} when the $i^{\text{th}}$--sub-system is active.
\end{defn}

-- with unbounded switching frequency which ICI stabilises the switched linear system. As the switching frequency increases the dynamics of the switched system approach those of the \emph{average system} $\dot x =  Ax$. This is as the increase in switching frequency minimises all terms in the Baker-Cambell-Haussdorff (BCH) formula \eqref{eq:bch2} which do not contribute to the convex combination. This switching function is then used to prove ICI stability of the switched affine system. The proposed switching function in \cite{town20} necessarily had unbounded switching frequency. The question of whether there exists a ICI stabilising switching function with a lower bound on the dwell-time was not addressed. The dynamics and stability of switched systems with different switching frequencies has also been addressed by \cite{libe22, libe24} and \cite{yang20}, amongst others.

Here, we address this and extend the results of \cite{town20} by establishing that if a system is ICI stabilisable then there must exist a periodic switching function which ICI stabilises the switched system. The extension is proved by considering the particular dynamics of an ICI stabilisable system under  periodic switching functions. In addition, we show that it is possible to ICI stabilise a switched system with a switching function having bounded switching frequency. The existence of such a function was not established in \cite{town20}.
We also show that, in general, under such a switching function an affine system will converge to a limit set but not to a point.

\section{Dwell time and the BCH formula}
\todo{to do -- do we need non-zero $\alpha_i$?}
We denote by $\alpha_i$ the \emph{normalised activation fraction} of each sub-system i.e.
\[
\alpha_i :=	\lim_{t \to \infty} \l ( \frac{a_i (t)}{t} \r ) \geq 0
\]%
for all $i$, where $a_i (t)$ is the measure of interval $[0,t]$ over which the $i^{\text{th}}$ sub-system is active. The fractions $\alpha_i$ correspond to the coefficients of the matrices $A_i$ and vectors $b_i$ of each sub-system of \eqref{eq:switchedaffine} in the convex combination. We note if $\alpha_i = 0$ then the system \eqref{eq:switchedaffine} is asymptotically the same with or without the $i^{\text{th}}$ sub-system.

Throughout, for a given switching signal $\sigma(t)$ with switching instances $(t_1, t_2, \dots)$ we denote by $\sigma(\eta, t)$ a switching signal with switching instants $(\eta t_1, \eta t_2, \dots)$. Moreover, for each $s_k \in (\eta t_i)$ and corresponding $t_k \in (t_i)$ we have that $\sigma(\eta,s_k) = \sigma(t_k)$ i.e. the order of activation is not affected by the choice of $\eta$.
As an example, if the system \eqref{eq:linsys} switches whenever $t \in \N$ under $\sigma(t)$ then \eqref{eq:linsys} switches whenever $2t \in \N$ under $\sigma(0.5, t)$.

We use the notion of \emph{dwell time} throughout. The dwell time of the $i^{\text{th}}$ sub-system is the length between two switching instants $t_k$ and $t_{k+1}$ over which the sub-system is active. As the switching function is an infinite sequence there are infinitely many dwell times, one for each instant. Unless otherwise specified we consider the average dwell time:
\[
	\mathbb{E} (\tau) :=  \lim_{T \to \infty} \l ( \frac{1}{T}  \r) \sum_{t_k < T} | t_{k} - t_{k-1}|
\]%
By abuse of notation we denote $\tau = \mathbb{E} (\tau)$. We observe that if a subs-system of \eqref{eq:switchedaffine} has dwell time $\tau_i$ under $\sigma(t)$ then it has dwell time $\eta \tau_i$ under $\sigma(\eta, t)$.

The \emph{Baker-Cambell-Haussdorff} formula, see \cite{ster09,town20}, is essential to our results. For two matrices $X, Y \in \R ^{n \times n}$ the Baker-Cambell-Haussdorff formula is a solution to the equation:
	$Z = \ln \l ( \exp(X) \exp(Y) \r )$. 
The solution, $Z$, may be expressed as a infinite sum in terms of the matrices $X, Y$ and terms involving nestings of their commutator $[X, Y]$. For example:
\begin{align*}
Z = X + Y + \frac{1}{2} [X, Y]+ \frac{1}{12} \l ( [X , [X,Y]] + [Y, [Y, X]] \r ) + \cdots
\end{align*}
Throughout we refer to the sum of all commutator-dependent terms as the commutator matrix $C$. Thus we write, $Z$, as:
\begin{align}
\label{eq:bch2}
	Z = X + Y + C
\end{align}%

\section{Periodic Signals}

 We say a switched system is periodic if the switching function $\sigma$ is periodic.

\begin{defn}[Periodic]
	$\sigma$ is periodic if there exists $T > 0$ such that $\sigma(t + nT) = \sigma(t)$ for all $t \in \mathbb{R}_+$ and $n \in \mathbb{N}$.
\end{defn}

 A switching sequence may be periodic with no lower bound on the dwell time $\tau$ of the sequence. Should $\tau \to 0$ then the results of \cite{town20} may be applied directly to such systems. Hence, the unique properties of periodic systems are most useful to study the dynamics of systems with non-zero dwell time i.e. systems for which the average dwell time $\mathbb{E} (\tau) > c >0$ for all $t$. We call such switching functions \emph{non-vanishing}.

\begin{defn}[Non-vanishing]
    We say $\sigma$ is \emph{non-vanishing} if there exists $\tau > 0$ such that for any discontinuity, $t'$, of $\sigma$, $\sigma(t)$ is constant for all $t \in [t',  t' + \tau)$.
\end{defn}

A periodic sequence is a sequence for which the order of activation of each sub-system is fixed and a non-vanishing periodic sequence is a sequence for which there exists $s$ such that $\sigma(t' + s) = \sigma(t')$ for all switching instants $t' > T$. As we consider convergence of the entire state-space we only need to consider the `eventual' properties of each switching function. For example should a switching sequence be aperiodic for all $t < T$ but periodic for all $t \geq T$, for some $T > 0$ then the stability under such a function is equivalent to the stability of the system under the function $\sigma(t)$ for $t \geq T$ i.e. we may assume that the sequence is periodic for all $t$. Throughout we will use the following definitions of the stability of a matrix.

\begin{defn}[Stable]
 A matrix $A$ is \emph{stable} if $Re( \lambda_i) < 0$ for all $\lambda_i \in eig(A)$ and is unstable if there exists $\lambda_i \in eig(A)$ such that $Re(\lambda_i) > 0$. 
\end{defn}

Furthermore, these definitions may be extended to the time-varying matrix $A(t)$. For example, if there exists $T$ such that $A(t)$ is stable for all $t\geq T$. In this case we say the matrix $A(t)$ is \emph{eventually} stable. We note that a -- potentially time-varying -- matrix $A(t)$ is \emph{not unstable} if $Re(\lambda_i) \leq 0$ for all $\lambda_i \in eig(A(t))$.

\Cref{thm:noname} establishes the existence of ICI stabilising periodic switching functions for any system which is ICI stabilisable.

\begin{thm}
\label{thm:noname}
    The following are equivalent:
    \begin{enumerate}
        \item{The system \eqref{eq:linsys} is ICI stabilisable.}
        \item{There exists a stable $A \in co ( \mathcal{A})$.}
        \item{The system \eqref{eq:linsys} is ICI stabilised by the norm-minimising signal.}
        \item{There is a periodic switching signal with non-zero dwell time which ICI stabilises \eqref{eq:linsys}.}
        \item{There is $M > 0$ such that for all dwell-times $\eta < M$ there is a switching signal which ICI stabilises \eqref{eq:linsys}.}
    \end{enumerate}
\end{thm}

\begin{proof}
\textbf{$(1 \iff 2 \iff 3)$}. By Theorem 9 of \cite{town20}, the system \eqref{eq:linsys} is ICI stabilisable if and only if it is ICI stabilised by the norm-minimising signal. Furthermore, Theorem 9 of \cite{town20} proves that the norm-minimising signal is ICI stabilising if and only if there exists a stable convex combination $A \in co ( \mathcal{A})$.

\textbf{$(2 \implies 4)$}. Let $(\alpha_1, \cdots, \alpha_m) \in \R ^{m}$ be co-efficients such that $A = \sum_i \alpha_i A_i$ is stable and define $T := \sum_i \alpha_i$. Furthermore, let $\sigma(\eta, t)$ be a periodic switching signal with period $\eta T$ such that if $t \in [0, \eta \alpha_1]$ then  $A_1$ is active, if $t \in (\eta \alpha_1,  \eta \alpha_2]$ then $A_2$ is active and so forth. Hence the state-transition matrix of \eqref{eq:linsys}, under $\sigma$, is
\begin{align*}
		\Phi(\eta, n \eta T) = \l ( e^{\eta \alpha_m A_m} \cdots e^{\eta \alpha_2 A_2} e^{\eta \alpha_1 A_1} \r ) ^n = \l (  \Phi ( \eta, \eta T) \r) ^n 
\end{align*}%
for all $n$. Using \eqref{eq:bch2} iteratively we have $\Phi(\eta, \eta T)$ is equal to
\begin{align}
\label{eq:usethis}
\begin{split}
\exp \l ( \eta \sum_{i \leq m} \alpha_i A_i  + \eta^2 \l ( \prod_{i\leq m} \alpha_i  \r ) C \r )  
              =  e^{ \l ( \eta A + o (\eta^2) C \r )}
              \end{split}
\end{align}

As $\Phi(\eta, t)$ is a continuous function of $\eta$ for any $\varepsilon > 0$ there exists $\eta$ such that
    $\l \|\Phi(\eta, \eta T) - \exp(\eta A) \r \| < \varepsilon$. 
Selecting $\varepsilon \in (0, 1 - \l \| e^{\eta A}\r \|)$ which is non-empty as $\l \| e^{\eta A} \r \| < 1$ we have that
    $\l \| \Phi( \eta, \eta T) \r \| < 1$. 

\textbf{$(2 \implies 5)$}. By the above there exists a periodic signal with state-transition matrix \eqref{eq:usethis} which gives the state of the system at the end point of the interval $[t, t + \eta T]$. Thus $\Phi(\eta, T)$ gives the state of the system at the endpoint of $[t, t + T]$ i.e. the duration elapsed does not depend on $\eta$ and over a duration $T$ the signal $\sigma(\eta, t)$ repeats $T/\eta$ times. Hence, assuming without loss of generality that $t = 0$ and that $T$ is divisible by $\eta$ 
\begin{align}
\l \| \Phi(\eta, T) - e^{ AT} \r \| &= \l \| e^{(\eta A + \eta^2 C)(T/\eta)}  - e^{AT} \r \| \nonumber\\
& = \l \| e^{( AT+ \eta C T) }  - e^{AT} \r \| \nonumber \\
& \leq \l \| \eta C T \r \| e^{ \l \| AT \r \|} e^{\l \| \eta CT\r \|}  \label{eq:inequp}
\end{align}
where the inequality follows as, for matrices $X, Y$, we have the identity   
    $\l \| e^{X + Y} - e^{X} \r \| \leq \l \| Y \r \|  e^{\l \|X \r \|} e^{\l \|Y \r \|}$. 
As $T$ is fixed the $\eta$ dependent terms tend to $0$ monotonically as $\eta$ tends to $0$.
Hence if there exists $\eta_1$ such that the upper bound \eqref{eq:inequp} is equal to some $\varepsilon > 0$ then the state-transition matrix $\Phi(\eta, T)$ must also be within the same $\varepsilon$-neighbourhood of $e^{ AT}$ for all $\eta \leq \eta_1$. Stability follows as $A$ is stable by assumption.

We note that $5$ and $4 \implies 1$ by definition. \hfill $\blacksquare$
\end{proof}

\subsection{An alternative proof of \Cref{thm:noname}}

Lemmas 2.10 and 2.11 from \cite{sun06} may also be used to prove the existence of a ICI stabilising switching signal with sufficiently small dwell time for \eqref{eq:linsys}. This is because Lemma 2.10 of \cite{sun06}, proves that systems with sufficiently small average dwell time, $\tau$, may be approximated by a periodic system within $o (\tau^2)$.  For a linear periodic system of the form \eqref{eq:linsys}, \cite{sun06} define the \emph{average system} as the time-invariant system:
\begin{align}
	\label{eq:averagesys}
	\dot x = A x
\end{align}%
where $A := \sum _i \alpha_i A_i$, with $\alpha_i$ being the fraction of $T$ for which the $i^{\text{th}}$ sub-system is active. The state-transition matrix of \eqref{eq:averagesys} is
	$\Psi(t,0) := \exp(A t)$. 

\begin{lem}[Lemma 2.11. \cite{sun06}]
	\label{lem:sun06}
	Suppose that the average system satisfies:
		$\Vert \Psi (t, t_0 ) \Vert \leq \beta e^{\delta (t - t_0)}$ 
		for all $t \geq t_0$ and for some $\beta >0$ and $\delta \in \R$. Then, for any $\varepsilon > 0$ there exist positive real numbers $\kappa$ and $\rho$ such that:
			$\Vert \Phi (t, t_0) \Vert \leq \kappa e^{(\delta + \varepsilon)(t - t_0)}$ 
		for all $t \geq t_0$ with $T \leq \rho$, where $\Phi(t, t_0)$ is the state-transition matrix for \eqref{eq:linsys} under a periodic signal, with period $T$.
\end{lem}

Hence we give the following alternative proof that $(2) \implies (4)$ from \Cref{thm:noname}.

\begin{proof}
By Theorem 9 of \cite{town20}, there must exist a stable convex combination $A \in \mathcal{A}$. Let $\delta < 0$ be its rate of convergence. By \Cref{lem:sun06} for any $\varepsilon > 0$ there must exist $\rho$ such that the periodic system satisfies:
	$\Vert \Phi(t) \Vert \leq \kappa e ^{(\delta + \varepsilon) t}$ 
for all $T \leq \rho$. Choosing $\varepsilon < - \delta$ gives that the periodic system is ICI stable.\hfill $\blacksquare$
\end{proof}

\subsection{Shift Invariance and Commutativity}

\def\ss{ICI stabilising}

As a signal is only ICI stabilising if it is eventually ICI stabilising, it suffices to only consider the behaviour of the system for all $t$ greater than a given finite $T$. As such any ICI stabilising signal remains ICI stabilising for any initialisation time i.e. it is shift-invariant.

\begin{lem}
\label{lem:firsttry}
 	Suppose there exists a periodic ICI stabilising switching function $\sigma$ for \Cref{eq:linsys}. Then the shifted switching function $\sigma(t +\gamma)$ ICI stabilises the switched linear system for all $\gamma \geq 0$.
\end{lem}

\begin{proof}
Let $P(t) := \Phi(t, \sigma(t))$ and $G(t):=\Phi(t, \sigma(t+ \gamma))$ be the state-transition matrices over the interval $[0, t]$ of \eqref{eq:linsys} under the unshifted switching signal $\sigma(t)$ and the shifted switching signal $\sigma(t + \gamma)$ respectively.
Define the matices $F(\gamma) := P^{-1} (\gamma)$ and $H(t, \gamma) := P(t + \gamma) P^{-1} (t)$ which, respectively, are the inverse of the state-transition matrix $P(t)$ over the interval $[0, \gamma]$ and the state transition matrix $\Phi(t, \sigma(t))$ over the interval $[t, t+ \gamma]$. Then $G(t) = P(t + \gamma) F(\gamma) = H(t + \gamma) P(t) F( \gamma)$ for all $\gamma$ Furthermore, for given $\gamma \geq 0$, $\l \| F(\gamma)\r \|$ and $\l \| H(t, \gamma) \r \|$, are bounded for all $t \geq \gamma$. Thus, $G(t) = P(t) + o(t)$, for all $t \geq \gamma$.\hfill $\blacksquare$
\end{proof}

Under a periodic signal the system \eqref{eq:linsys} is stable if and only if $ | det \l ( \Phi(\eta, T) \r )| < 1$. As $det(XY) = det(YX) = det(X) det(Y)$, stability is maintained under any activation order of the sub-systems so long as each sub-system is activated for the same duration. We formalise this in \Cref{lem:det}.

\begin{lem}
\label{lem:det}
    The system \eqref{eq:linsys} is ICI stabilised by the periodic switching signal $\sigma(\eta, t)$ if and only if it is stabilised by any permutation $\sigma'(\eta, t)$ of the signal.
\end{lem}

\begin{proof}
This follows as the determinant of the product of matrices is commutative.\hfill $\blacksquare$
\end{proof}

We note that the individual eigenvalues and eigenvectors are not preserved under the permutation -- unless the permutation is a cycle -- of a switching signal $\sigma$ but only their product. So $\lambda_i \in eig(\Phi(\sigma, t))$ does not imply that $\lambda_i \in eig(\Phi( g \circ \sigma, t) )$ for a permutation $g \in \mathbb{Z}_m$ where $\Phi(\sigma, t)$ is the state-transition matrix of \eqref{eq:linsys} under $\sigma$. The fact that the eigenvectors are not preserved under permutation follows as $v$ is an eigenvector of the product $AB$ i.e. $AB v = \lambda v$ if and only if $Bv$ is an eigenvector of $BA$ with eigenvalue $\lambda$.

In \Cref{lem:itdoesntexistyet} we give a conservative implicit upper bound for $\eta$, in terms of the commutator matrix $C$, which ensures that the periodic non-vanishing switching signal $\sigma(\eta', t)$ is ICI stabilising for all $\eta' \leq \eta$, provided that $\sum_i \alpha_i A_i$ is stable for the  normalised activation times $\alpha_i$ of~$\sigma$.

In the proof of \Cref{lem:itdoesntexistyet}, we use the \emph{Lie Product} formula, \eqref{eq:lieprod}, \citep{tauv05} i.e. for matrices $M_i$
\begin{align}
	\label{eq:lieprod}
    \exp \l ( \sum M_i \r ) = \lim_{k \to \infty} \l ( \prod \exp \l ( \frac{1}{k}M_i \r ) \r ) ^k
\end{align}

\begin{lem}
\label{lem:itdoesntexistyet}
Suppose $A \in co \l ( \mathcal{A} \r )$ is stable. Then the periodic non-vanishing switching signal $\sigma(\eta, t)$ ICI stabilises \eqref{eq:linsys} if, for all $k$,
\begin{align}
\label{eq:uselessineq}
    \l \| \exp \l ( \frac{\eta ^2}{k} \l( \prod_i \alpha_i \r)  C\r )\r \| < \l \| \exp \l ( \frac{\eta}{k} \sum_i \alpha_i A_i \r )\r \|^{-1}
\end{align}%
\end{lem}

\begin{proof}
The state-transition matrix under $\sigma(\eta, t)$, after one period $T$, is given by
\[
    \Phi(\eta, T) = \exp \l ( \eta \sum_{i = 1} ^m \alpha_i A_i + \eta^2 \prod_{i=1} ^m \alpha_i C \r) 
\]%
Hence
\begin{align*}
    \l \| \Phi (\eta, T) \r \| &= \lim_{k \to \infty} \l \|  \l ( e^{ \l ( \frac{\eta}{k} \sum_{i = 1} ^m \alpha_i A_i \r)}     e^ {\l (  \frac{\eta^2}{k} \prod_{i=1} ^m \alpha_i C \r) }\r ) ^k \r \|\\
                               &\leq \lim_{k \to \infty} \l \|  \l ( e^{ \l ( \frac{\eta}{k} \sum_{i = 1} ^m \alpha_i A_i \r)}     e^ {\l (  \frac{\eta^2}{k} \prod_{i=1} ^m \alpha_i C \r) }\r ) \r \|^k\\
                               &\leq \lim_{k \to \infty} \l \|   e^{ \l ( \frac{\eta}{k} \sum_{i = 1} ^m \alpha_i A_i \r)} \r \|^k    \l \| e^ {\l (  \frac{\eta^2}{k} \prod_{i=1} ^m \alpha_i C \r) } \r \|^k
\end{align*}
which is less than $1$ by assumption. Hence, the system is ICI stabilised by $\sigma(\eta, t)$.\hfill $\blacksquare$
\end{proof}

As $\l \| C\r\|$ is bounded and constant there must always exist $\eta > 0$ such that \eqref{eq:uselessineq} is satisfied.

\section{Affine Systems}

    Given a switched affine system of the form \eqref{eq:switchedaffine} with sub-systems $\l ( \mathcal{A}, \mathcal{B} \r ) = \l \{ \l ( A_1, b_1 \r ), \l ( A_2, b_2 \r ), \cdots, \l ( A_m, b_m \r ) \r \}$ we say the \emph{underlying} linear system is the switched linear system \eqref{eq:linsys} with sub-systems $\mathcal{A} = \l \{ A_1, \cdots, A_m \r \}$ under the same switching function $\sigma(t)$.

 Intuitively, we may separate an affine system into a `drift' arising from the addition of $b_i$ and a linear system $\dot x = A_{\sigma(t)} (x)$. Thus if the linear system becomes unbounded then the solution to the affine system remains bounded only if the `drift' counters the divergence of the linear system. This cannot occur, as should the drift ensure that the solution for initial condition $x(0)$ is bounded then it increases the rate of divergence for the solution when the initial condition is $-x(0)$. Hence, ICI stability of the underlying linear system is necessary for ICI stability of a switched affine system i.e. for the switched affine system to converge to an equilibrium point. We prove this in \Cref{thm:omgname} for a system where all sub-systems have a common equilibrium.

As noted by \cite{egid21}, in general, a switched affine system has multiple equilibria meaning it is not possible to ICI stabilise a switched affine system with a switching function which has non-zero average dwell time. However it is possible to \emph{practically ICI stabilise} the system which is equivalent to ICI stability of the underlying switched linear system. For such systems the ICI stabilisability is necessary and sufficient for the practical ICI stabilisability of \eqref{eq:switchedaffine}.

\begin{thm}
\label{thm:omgname}
    Suppose $\sigma$ is a non-vanishing switching signal which ICI stabilises \eqref{eq:switchedaffine}. Then each sub-system of \eqref{eq:switchedaffine} has a common equilibrium point.
\end{thm}

\begin{proof}
    Suppose $\sigma$ ICI stabilises \eqref{eq:switchedaffine} i.e. there exists a unique $y \in \R^n$ such that
    \[
        y := \lim_{t \to \infty} \Phi(t, \sigma) x(0)
    \]%
    for all initial conditions $x(0)$ where $\Phi(t, \sigma)$ is the state-transition matrix at $t$ for \eqref{eq:switchedaffine} under $\sigma$. This implies that $y = \Phi(t, \sigma) y$ for all $t$. Hence $y$ must be an equilibrium point for each sub-system.\hfill $\blacksquare$
\end{proof}

\begin{eg}
\label{eg:one}
We consider the two sub-system switched affine system with sub-system matrices
\begin{align}
\label{eq:ega}
	A_1  :=  \begin{pmatrix}
		-2.1 & -2 \\
		0.5 & 1
		\end{pmatrix}
		\text{ and }
	A_2 := \begin{pmatrix}
		1 & 2 \\
		0.1 & -2
		\end{pmatrix}
\end{align}
and affine components
		\[ b_1 := \begin{pmatrix}
				-2 & 
				1
			\end{pmatrix}'
			\text{ and }
			b_2 := \begin{pmatrix}
			 2  &
			 -2
			 \end{pmatrix}'
			 \]%
Each of the sub-systems share a common equilibrium point $e = -A_1^{-1} b_1 = -A_2 ^{-1} b_2  = (0 ,-1)$. The system is ICI stabilised by the switching function:
\begin{align}
\label{eq:signal}
\sigma(\eta, t) :=  \begin{cases}
			1, &\lfloor t/ \eta \rfloor \; mod(4) \leq 1 \\
			2, & \text{otherwise}
			\end{cases}%
\end{align}%
for all sufficiently small $\eta$. We note that under this switching function $\alpha_1 = \alpha_2 = 0.5$ i.e. both systems are sequentially activated for the same duration. 

\begin{figure}[H]
\begin{center}
\includegraphics[scale=0.16]{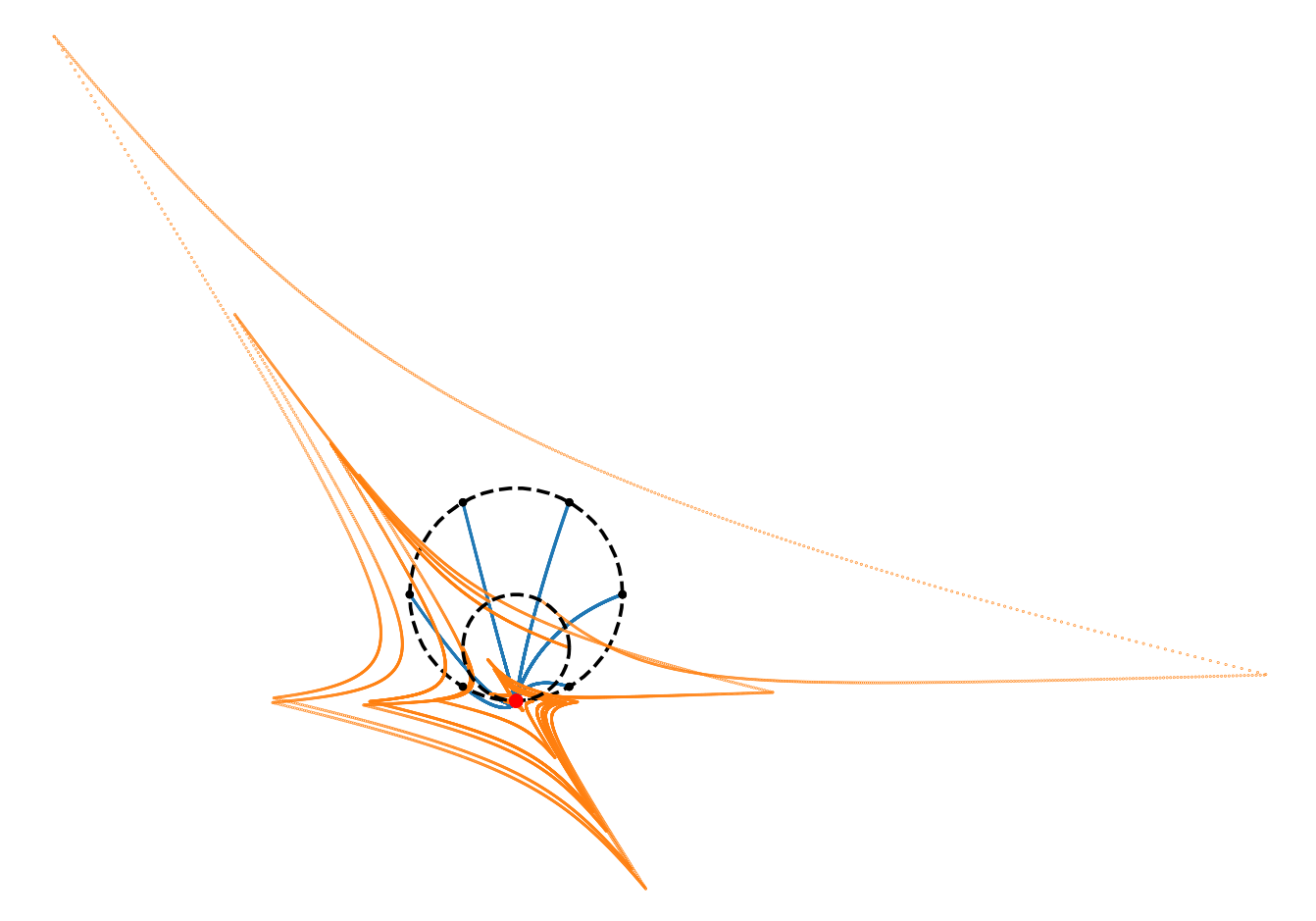}
\end{center}
\caption{The convergence of system \eqref{eq:switchedaffine} to the common equilibrium point $(0, -1)$ under $\sigma(\eta t)$ for $\eta = 1.1$ (orange) and $\eta = 10^{-3}$ (blue). The value of $\eta = 10^{-3}$ approximates the average system. The initial set was taken to be the unit ball centred at the origin $B(0)$. This has been scaled by $0.5$ when $\eta = 1.1$ for presentation purposes.}
\label{fig:putitthere}
\end{figure}%

The convergence of the system to the equilibrium $(0, -1)$ from the unit ball under $\sigma(\eta,t)$ is shown in \Cref{fig:putitthere}. The paths of convergence under $\sigma(1.1, t)$ for a few selected points on the boundary of the unit ball are shown by the orange traces in \Cref{fig:putitthere}. The paths of the solution to the average system which is approximated by the switching function $\sigma(10^{-3}, t)$, from the same initial points (marked on a scaled circle), are shown by the blue traces.
\end{eg}

\subsection{Practical Stability}

Whenever each sub-system in \eqref{eq:switchedaffine} does not share a common equilibrium point, if $\sigma$ is periodic and non-vanishing then it does not ICI stabilise \eqref{eq:switchedaffine}. However the system is \emph{ICI practically stabilisable} whenever \eqref{eq:linsys} is ICI stabilisable. 

\begin{defn}[Practical stability]
	The system \eqref{eq:switchedaffine} is ICI practically stabilisable if there exist $\sigma$ and bounded $P \subset \mathbb{R}^n$ such that $x(t) \to x \in P$ for all initial conditions $x(0)$ and $x(t) \in P$ for all $t \geq T$ if $x(T) \in P$.
\end{defn}

In \Cref{thm:nameit}, we use the average affine system 
\begin{align}
\label{eq:avaffine}
    \dot x = A x + b = \l ( \sum_i \alpha_i A_i \r)x + \sum_i \alpha_i b_i
\end{align}
which has an equilibrium at $x = -A^{-1} b$.

\begin{thm}
\label{thm:nameit}
    The following are equivalent
    \begin{enumerate}
        \item{The signal $\sigma$ ICI stabilises the system \eqref{eq:linsys}.}
        \item{The signal $\sigma(\eta, t)$ ICI stabilises the average affine system for all $\eta \leq 1$.}
        \item{There exists a non-vanishing periodic signal $\sigma(\eta, t)$ which  ICI stabilises the average affine system for all $\eta \leq 1$.}
        \item{The signal $\sigma$ ICI practically stabilises the system \eqref{eq:switchedaffine}.}
        \item{There is a non-vanishing periodic switching signal $\sigma(\eta, t)$ which ICI practically stabilises the system \eqref{eq:switchedaffine} for all $\eta \leq 1$.}
        \item{There is a periodic signal $\sigma$ which ICI stabilises the system \eqref{eq:linsys}.}
    \end{enumerate}
\end{thm}

\begin{proof}
\textbf{$(1 \implies (2 \text{ and } 3) )$}
\todo{to do -- follows by \Cref{thm:name}?}
    Suppose $\sigma$ ICI stabilises \eqref{eq:linsys}. This implies that $A := \sum_i \alpha_i A_i$ is stable. Hence, making the variable substitution $x = z - y$ where $y$ is the equilibrium point of the average system, we have that the solution $z(t)$ to $\dot z = A z$ converges to $0$ for all initial conditions $z(0)$ i.e. $x = y$. Thus the average affine system is ICI stabilised by $\sigma$. Furthermore, by \Cref{thm:noname} there exists a non-vanishing periodic signal $\sigma(\eta, t)$ which ICI stabilises \eqref{eq:linsys} for all $\eta \leq 1$ and hence, by the above argument ICI stabilises the average system for all $\eta \leq 1$.
    
\textbf{$(2 \text{ and } 3 \implies 4 \text{ and } 5)$}
    Let $(t_i)$ be the sequence of switching instances of the signal $\sigma(\eta, t)$. On the interval $(t_{i-1},t_i)$ whilst the $i^{\text{th}}$ sub-system is active we have for system \eqref{eq:switchedaffine}
    \[
        x(t) - x(t_{i-1}) = (t-t_{i-1}) \l( A_i x(t_{i-1}) + b_i \r ) + o(\tau_i^2)
    \]%
    where $\tau_i := t_{i} -t_{i-1}$ and $t\in (t_{i-1},t_i)$. Thus after one switching cycle with duration $T$ we have
    \[
        x(t_m) - x(t_0) =  \l ( \sum_{i=1} ^m \tau_i \l ( A_i x(t_{i-1}) + b_i \r ) \r ) + o(\tau ^2)
    \]%
    where  $\tau := \max \{ \tau_i\} \leq T$. Thus
    \begin{align*}
        \frac{x(t_m) - x(t_0)}{T} & = \l ( \sum_{i=1} ^m  \l ( \frac{\tau_i}{T} \r ) \l ( A_ix(t_{i-1}) + b_i \r ) \r) + o(\tau^2)  \\
        &= \l ( \sum_i \alpha_i \l ( A_ix(t_{i-1}) + b_i \r ) \r) + o(\tau ^2)
    \end{align*}%
    which converges to $\dot x = Ax + b$ as $\eta \to 0$ and therefore $T \to 0$. Hence for all $\varepsilon > 0$ there exists a sufficiently small maximum dwell-time $\tau$ so that
    \[
        \l \| \hat x(t) - x(t) \r \| < \varepsilon
    \]%
    for all $t$ where $\hat x(0) = x(0)$ and $\hat x$ is the solution under the average system and $x$ is the solution of \eqref{eq:switchedaffine} under the non-vanishing periodic signal $\sigma(\eta, t)$.
    \todo{to do -- might be better to show b'dedness of gen. soln?}

\textbf{$(5 \implies 6)$}
    Suppose there exists $\sigma(\eta, t)$ which ICI practically stabilises \eqref{eq:switchedaffine}. After $n$ switching instances the solution of \eqref{eq:switchedaffine} is
     \begin{align}
     \label{eq:gensoln}
     x(t_n) &= \l ( \prod_{i=1} ^n e^{A_i \tau_i } \r ) x(0) + \sum_{i = 1} ^n \l ( \prod_{j = i} ^n e^{A_j \tau_j } \r ) I_i
 	\end{align}
    where $\tau_i := t_i - t_{i-1}$ for all $i > 0$, $\tau_j$ is defined similarly,  $I_i=\int_{t_{i-1}}^{t_i} e^{-A_i \tau} b_i d\tau $ and we note that if $n > m$ then the $A_i$ repeat after $A_m$ e.g. $A_{m+1} = A_1$ and $A_{m+2} = A_2$. Hence, there exists $M$ such that 
    
    \begin{align}
    \label{eq:abound}
        \l \| \prod_{i = 1} ^{n} e^{A_i \tau_i } x(0) + \sum_{i = 1} ^{n} \l ( \prod_{j=i} ^{n} e^{A_j \tau_j}\r ) I_i  \r \| < M
    \end{align}%
    for all $n$ and for all $x(0)$. Thus in particular \eqref{eq:abound} holds for $x(0)$ and $-x(0)$. Thus
    \begin{align}
    \label{eq:abound1}
        \l \|- \prod_{i = 1} ^n e^{A_i \tau_i } x(0) + \sum_{i = 1} ^n \l ( \prod_{j=i} ^n e^{A_j \tau_j}\r ) I_i  \r \| < M
    \end{align}%
   Making the variable substitutions 
   \[
   Y := \prod_{i = 1} ^n e^{A_i \tau_i} x(0) \text{ and } X := \sum_{i = 1} ^n \l ( \prod_{j=i} ^n e^{A_j \tau_j}\r ) I_i
   \]
   by the parallelogram identity we have
   \begin{align}
   \label{eq:pty}
        2 \l \| X \r \|^2 + 2 \l \| Y \r \|^2 = \l \| X -Y \r \|^2 + \l \| X + Y \r \|^2 < 2M
   \end{align}%
   which holds if and only if $\l \| X \r \| \leq M$ and $\l \| Y \r \| \leq M$. This implies that 
       $Y_n := \prod_{i = 1} ^n e^{A_i \tau_i}$ 
   must be not unstable for all sufficiently large $n$. However, if $Y_n$ is not `strictly' stable, i.e. there exists an eigenvalue $\lambda_i$ of $Y_n$ such that $\|\lambda_i  \| = 1$, for all $n$ then there exists $M'$ such that the ball $B_{M'} (0)$  satisfies $\l \| Y_n B_{M'} (0) \r \| > M$ for all $n$ which contradicts \eqref{eq:pty}.
   We note that $(6)$ implies $(1)$ by \Cref{thm:noname}.\hfill $\blacksquare$
\end{proof}

\begin{cor}
Suppose $\sigma$ ICI stabilises \eqref{eq:linsys}. Then $\sigma$ ICI stabilises \eqref{eq:switchedaffine} only if either each sub-system has a common equilibrium point or the signal $\sigma$ is vanishing.
\end{cor}

\begin{proof}
    Suppose $\sigma$ ICI stabilises \eqref{eq:linsys} and that $y$ is a common equilibrium point for each sub-system of \eqref{eq:switchedaffine}. Making the variable substitution $x = z - y$ we have that each sub-system of \eqref{eq:switchedaffine} is now of the form
        $\dot z = A_i z$. 
    Thus after $N$ switching instances the solution of \eqref{eq:switchedaffine} is
    \begin{align}
    \label{eq:fornotation}
        z(t_N) = \prod_{i = 1} ^N e^{A_i (t_i - t_{i-1})} z(0)
    \end{align}
    which converges to $0$ for all initial conditions $z(0)$ as $\sigma$ ICI stabilises \eqref{eq:linsys}. Hence, $\lim_{N \to \infty} x(t_N) = y$. We emphasise that the index $i$ in \eqref{eq:fornotation} refers to the switching instant and not the index of the matrix i.e. $A_i \in \mathcal{A}$ for all $i$ even if $i > m$. Now, suppose the system is in equilibrium and the $i^{\text{th}}$ sub-system is active. This implies:
	$0 = A_i y + b_i$ i.e. $A_i y = - b_i$. 
	At the next switching instant \eqref{eq:switchedaffine} becomes
	$\dot y = A_{k} y + b_{k}$ 
	which is in equilibrium only if $A_{k} y = -b_{k}$.  Hence \eqref{eq:switchedaffine} remains in equilibrium if and only if $A_k y = -b_k$ for all $k$ i.e. $y$ is an equilibrium point of each sub-system. So once the state of \eqref{eq:switchedaffine} equals $y$ it remains there independent of the switching signal.

    Suppose instead that $\sigma$ is vanishing. By the proof of \Cref{thm:nameit}, we have that the system arbitrarily approaches the average affine system \eqref{eq:avaffine} i.e. for all $\varepsilon > 0$ there exists $T$ such that
        $\l \| \hat x (t) - x(t)\r \| < \varepsilon$ 
    for all $t \geq T$ and all initial conditions $x(0)$, where $\hat x(t)$ is the solution \eqref{eq:avaffine} and $x(t)$ is the solution to \eqref{eq:switchedaffine} under $\sigma$. As $A$ is stable because $\sigma$ ICI stablises \eqref{eq:linsys} we have that $\hat x (t)$, and therefore $x(t)$, approaches the equilibrium point $-A^{-1} b$.\hfill $\blacksquare$
\end{proof}

These results bound the region of attainable equilibria for an affine switched system in terms of the switching frequency and the dynamics of the average system. This relates to the design of switching functions to attain specific equilibria as studied by \cite{egid21}.

\begin{eg}

In \Cref{fig:putputitmore} we consider the affine switched system with the same sub-system matrices, \eqref{eq:ega}, and switching signal, \eqref{eq:signal}, as in \Cref{eg:one}. However the affine components are
        $b_1 := \begin{pmatrix}
				-2 & 
				1
			\end{pmatrix}'
			\text{ and }
			b_2 := \begin{pmatrix}
			 2 & 
			 2
			 \end{pmatrix}'$.
\begin{figure}[H]
\begin{center}
\includegraphics[scale=0.16]{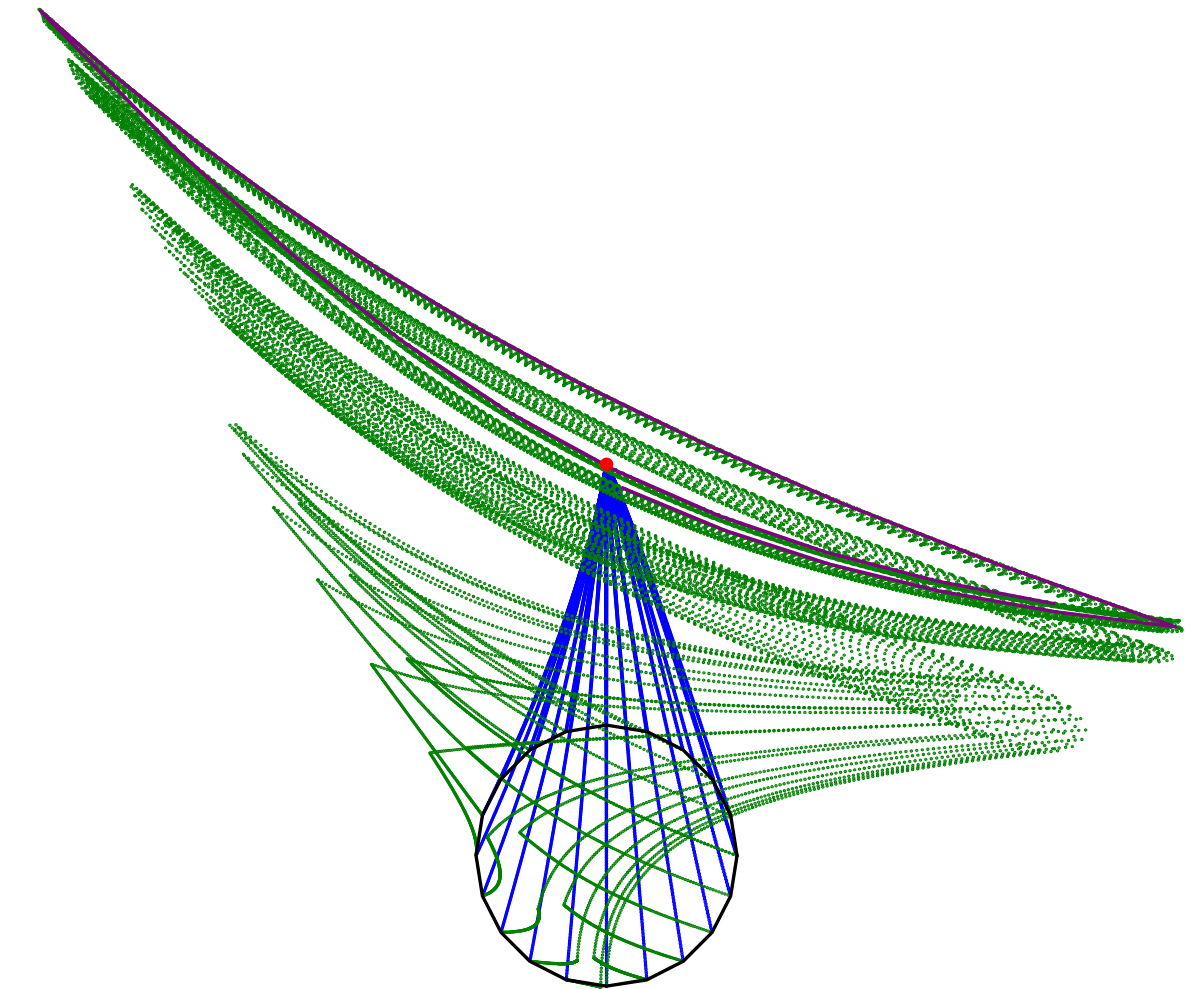}
\end{center}
\caption{The convergence of \eqref{eq:switchedaffine} under $\sigma(0.5, t)$ (green) and of the average system (blue). The average system is stable converging to the equilibrium $(0,3)$ (red). For all sufficiently small $\eta > 0$ the system converges to a closed orbit (purple). The initial set was $B(0)$.}
\label{fig:putputitmore}
\end{figure}
Hence the sub-systems do not share a common equilibrium point. The equilibria for each sub-system are $e_1 = -A_1 ^{-1} b_1=(0,-1)$ and $e_2 = - A_2 ^{-1} b_2 = (-3.64, 0.82)$. Thus the average system -- which corresponds to the system under the switching signal $\lim_{\eta \to 0} \sigma(\eta, t)$ where $\sigma(\eta, t)$ is defined in \eqref{eq:signal} -- has an equilibrium at $(0,3)$. This is shown by the red point in \Cref{fig:putputitmore}. The solutions of the switched affine system under the signal $\sigma(0.5, t)$ from a number of initial conditions on the boundary of the unit ball are shown by the green trajectories. The solutions of the average affine system from the same initial conditions are shown by the blue trajectories. Lastly the purple closed orbit represents the limit cycle of the system under $\sigma(0.5, t)$.
\end{eg}

\section{Conclusions}
Using periodic systems, we have proven that there exist switching functions which practically ICI stabilise switched affine systems with bounded switching frequency whenever the system is practically ICI stabilisable. We have also proven that, if the switching frequency is unbounded or each sub-system of a switched affine system shares a common equilibrium -- linear systems are a special case when this common equilibrium is $0$ -- then the system is ICI stabilisable to the equilibrium of the average affine system.

\bibliography{references} 

\begin{thebibliography}{15}
\providecommand{\natexlab}[1]{#1}
\providecommand{\url}[1]{\texttt{#1}}
\providecommand{\urlprefix}{URL }
\expandafter\ifx\csname urlstyle\endcsname\relax
  \providecommand{\doi}[1]{doi:\discretionary{}{}{}#1}\else
  \providecommand{\doi}{doi:\discretionary{}{}{}\begingroup \urlstyle{rm}\Url}\fi

\bibitem[{Deaecto et~al.(2010)Deaecto, Geromel, Garcia, and Pomilio}]{deae10}
Deaecto, G.S., Geromel, J.C., Garcia, F.S., and Pomilio, J.A. (2010).
\newblock Switched affine systems control design with application to dc–dc converters.
\newblock \emph{IET Control Theory and Applications}, 4(7), 1201 -- 1210.

\bibitem[{Egidio and Hansson(2021)}]{egid21}
Egidio, L.N. and Hansson, A. (2021).
\newblock On the search for equilibrium points of switched affine systems.
\newblock \emph{IFAC-PapersOnLine}, 54(5), 301--306.
\newblock 7th IFAC Conference on Analysis and Design of Hybrid Systems ADHS 2021.

\bibitem[{Feron(1996)}]{fero96}
Feron, E. (1996).
\newblock {Quadratic Stabilizability of Switched Systems via State and Output Feedback}.
\newblock Technical report, Center for Intelligent Control Systems, MIT, Cambridge, MA.

\bibitem[{Geromel and Colaneri(2006)}]{gero06}
Geromel, J.C. and Colaneri, P. (2006).
\newblock Stability and stabilization of continuous-time switched linear systems.
\newblock \emph{SIAM Journal on Control and Optimization}, 45(5), 1915--1930.

\bibitem[{Khalil(1992)}]{khal92}
Khalil, H.K. (1992).
\newblock \emph{Nonlinear systems}.
\newblock Macmillan Publishing Company, New York.

\bibitem[{Liberzon(2003)}]{libe03}
Liberzon, D. (2003).
\newblock \emph{Switching in systems and control}.
\newblock Springer Science \& Business Media.

\bibitem[{Liberzon and Shim(2022)}]{libe22}
Liberzon, D. and Shim, H. (2022).
\newblock Stability of linear systems with slow and fast time variation and switching.
\newblock In \emph{2022 IEEE 61st Conference on Decision and Control (CDC)}, 674--678.

\bibitem[{Liberzon and Shim(2024)}]{libe24}
Liberzon, D. and Shim, H. (2024).
\newblock Stability of nonlinear systems with slow and fast time variation and switching: the common equilibrium case.
\newblock In \emph{2024 IEEE 63rd Conference on Decision and Control (CDC)}, 5622--5627.

\bibitem[{Lin and Antsaklis(2009)}]{lin09}
Lin, H. and Antsaklis, P.J. (2009).
\newblock Stability and stabilizability of switched linear systems: a survey of recent results.
\newblock \emph{IEEE Transactions on Automatic control}, 54(2), 308--322.

\bibitem[{Lu and Yang(2020)}]{yang20}
Lu, A.Y. and Yang, G.H. (2020).
\newblock Stabilization of switched systems with all modes unstable via periodical switching laws.
\newblock \emph{Automatica}, 122, 109150.

\bibitem[{Shorten et~al.(2007)Shorten, Wirth, Mason, Wulff, and King}]{shor07}
Shorten, R., Wirth, F., Mason, O., Wulff, K., and King, C. (2007).
\newblock Stability criteria for switched and hybrid systems.
\newblock \emph{SIAM review}, 49(4), 545--592.

\bibitem[{Sternberg(2009)}]{ster09}
Sternberg, S. (2009).
\newblock \emph{Lie Algebras}.
\newblock University Press of Florida.

\bibitem[{Sun(2006)}]{sun06}
Sun, Z. (2006).
\newblock \emph{Switched linear systems: control and design}.
\newblock Springer Science \& Business Media.

\bibitem[{Tauvel and Rupert(2005)}]{tauv05}
Tauvel, P. and Rupert, W. (2005).
\newblock \emph{Lie algebras and algebraic groups}.
\newblock Springer.

\bibitem[{Townsend and Seron(2020)}]{town20}
Townsend, C. and Seron, M.M. (2020).
\newblock Existence of initial condition independent stabilising switching functions.
\newblock In \emph{2020 59th IEEE Conference on Decision and Control (CDC)}, 3273--3278.

\end{thebibliography}
\end{document}